\documentclass[12pt,english,BCOR7.5mm]{amsart}
\usepackage{ae,aecompl}
\usepackage[T1]{fontenc}
\usepackage[latin9]{inputenc}
\usepackage[letterpaper]{geometry}
\usepackage{color}
\usepackage{babel}
\usepackage{amsthm}
\usepackage{amssymb}
\usepackage[unicode=true,
 bookmarks=true,bookmarksnumbered=true,bookmarksopen=true,bookmarksopenlevel=2,
 breaklinks=false,pdfborder={0 0 1},backref=false,colorlinks=true]
 {hyperref}
\hypersetup{pdftitle={Using XY-pc in LyX},
 pdfauthor={H. Peter Gumm},
 pdfsubject={LyX's XY-pic manual},
 pdfkeywords={LyX, documentation},
 linkcolor=black, citecolor=black, urlcolor=blue, filecolor=blue,pdfpagelayout=OneColumn, pdfnewwindow=true,pdfstartview=XYZ, plainpages=false, pdfpagelabels}

\makeatletter
\numberwithin{equation}{section}
\numberwithin{figure}{section}
\theoremstyle{plain}
\newtheorem{thm}{\protect\theoremname}
  \theoremstyle{definition}
  \newtheorem{problem}[thm]{\protect\problemname}
  \theoremstyle{definition}
  \newtheorem{defn}[thm]{\protect\definitionname}
  \theoremstyle{plain}
  \newtheorem{lem}[thm]{\protect\lemmaname}
  \theoremstyle{plain}
  \newtheorem{prop}[thm]{\protect\propositionname}

\usepackage[all]{xy}

\newcommand{\xyR}[1]{
  \xydef@\xymatrixrowsep@{#1}}
\newcommand{\xyC}[1]{
  \xydef@\xymatrixcolsep@{#1}}

\newdir{|>}{!/4.5pt/@{|}*:(1,-.2)@^{>}*:(1,+.2)@_{>}}

\let\myTOC\tableofcontents
\renewcommand\tableofcontents{%
  \pdfbookmark[1]{\contentsname}{}
  \myTOC }

\def\LyX{\texorpdfstring{%
  L\kern-.1667em\lower.25em\hbox{Y}\kern-.125emX\@}
  {LyX}}

\makeatother

  \providecommand{\definitionname}{Definition}
  \providecommand{\lemmaname}{Lemma}
  \providecommand{\problemname}{Problem}
  \providecommand{\propositionname}{Proposition}
\providecommand{\theoremname}{Theorem}

\begin{document}

\title{Geometric Progression-Free Sequences with Small Gaps II}

\author{Xiaoyu He}
\begin{abstract}
When $k$ is a constant at least $3$, a sequence $S$ of positive
integers is called $k$-GP-free if it contains no nontrivial $k$-term
geometric progressions. Beiglböck, Bergelson, Hindman and Strauss
first studied the existence of a $ $$k$-GP-free sequence with bounded
gaps. In a previous paper the author gave a partial answer to this
question by constructing a $6$-GP-free sequence $S$ with gaps of
size $O(\exp(6\log n/\log\log n))$. We generalize this problem to
allow the gap function $k$ to grow to infinity, and ask: for which
pairs of functions $(h,k)$ do there exist $k$-GP-free sequences
with gaps of size $O(h)$? We show that whenever $(k(n)-3)\log h(n)\log\log h(n)\ge4\log2\cdot\log n$
and $h,k$ satisfy mild growth conditions, such a sequence exists.
\end{abstract}
\maketitle

\section{Introduction}

Let $S$ be an increasing sequence of positive integers. We say that
$S$ is $k$-GP-free if it contains no $k$-term geometric progressions
with common ratio not equal to $1$, where $k\geq3$ for the problem
to be nontrivial. Let $h$ be a nondecreasing function $\mathbb{N}\rightarrow\mathbb{R}^{+}$.
We say that a sequence $S$ has gaps of size $O(h)$ if there exists
a constant $C>0$ such that for every pair $m,N\in\mathbb{N}$ with
$m\le N$, the sequence $S$ intersects the interval $[m,m+Ch(N))$. 

The maximal asymptotic density of a $k$-GP-free sequence is well-studied
\cite{BG,McNew,NO,Riddell}. Beiglböck et al. \cite{BBHS} originally
posed the related question:
\begin{problem}
\label{kconst}Does there exist $k\geq3$ and a $k$-GP-free sequence
$S$ such that $S$ has gaps of size $O(1)$?
\end{problem}
The standard example of a $3$-GP-free sequence is the sequence $Q$
of positive squarefree numbers $1,2,3,5,6,7,10,\ldots$, which has
asymptotic density $\frac{6}{\pi^{2}}$. Despite its large density,
the size of its largest gaps is not known. The best unconditional
result available is that of Filaseta and Trifonov \cite{FT} that
$Q$ has gaps of size $O(N^{1/5}\log N)$, and Trifonov also established
a generalization that the sequence of $k$-th-power-free numbers has
gaps of size $O(N^{1/(2k+1)}\log N)$ \cite{Trifonov}. Assuming the
$abc$ conjecture, Granville showed that the gaps of $Q$ are of size
$O(N^{\varepsilon})$ for all $\varepsilon>0$ \cite{Granville}. 

All of these bounds can be improved immensely if we assume the conjecture
of Cramér that the gaps between consecutive primes are $O(\log^{2}N)$
\cite{Cramer}. For a discussion of Cramér's model and implications,
see the article of Pintz \cite{Pintz}. The problem of bounding largest
gaps between consecutive primes, both from above and below, is notoriously
difficult, and the best known lower bound is
\[
p_{n+1}-p_{n}\ge\frac{C\log p_{n}\log\log p_{n}\log\log\log\log p_{n}}{\log\log\log p_{n}}
\]
for some $C>0$ and infinitely many $n$, due to Ford, Green, Konyagin,
Maynard, and Tao \cite{FGKMT}, an improvement by $\log\log\log p_{n}$
over the longstanding bound of Rankin \cite{Rankin2}. The best unconditional
upper bound is $p_{n+1}-p_{n}=O(N^{0.525})$, due to Baker, Harman,
and Pintz \cite{BHP}, with $O(N^{1/2}\log N)$ possible assuming
the Riemann hypothesis.

Instead of pursuing these notoriously difficult problems, in a previous
paper the author showed that by replacing $Q$ by a randomly constructed
analogue, we can improve on Granville's bound unconditionally.
\begin{thm}
\label{thm:prev}\cite{He} There exists a $6$-GP-free sequence $T$
and a constant $C>0$ such that the gaps of $T$ are of size $O(\exp(C\log N/\log\log N)).$
In fact $C$ can be taken to be any positive real greater than $\frac{5}{6}\log2$.
\end{thm}
In this paper we generalize the Problem \ref{kconst} as follows.
Henceforth $k$ is no longer a constant but a nondecreasing function
$k:\mathbb{N}\rightarrow\mathbb{R}_{\geq3}$. We say that $S$ is
$k$-GP-free if for every $N\in\mathbb{N}$, the finite subsequence
$S\cap\{1,2,\ldots,N\}$ does not contain any nontrivial geometric
progressions of length at least $k(N)$.
\begin{problem}
For which pairs of functions $(h,k)$ do there exist $k$-GP-free
sequences $S$ such that $S$ has gaps of size $O(h)$?
\end{problem}
We call $h$ the gap function and $k$ the length function, and a
pair $(h,k)$ feasible if such an $S$ exists. Thus far we have only
dealt with constant length function; in particular Theorem 2 shows
that the pair $(\exp(C\log N/\log\log N),6)$ is feasible. At the
other end of the spectrum, it is trivial that $(1,\log N/\log2)$
is a feasible pair, simply because the longest possible geometric
progression in $1,\ldots,N$ has length at most $\log N/\log2$. In
the last section of this paper we show in fact that $(1,\varepsilon\log N)$
is feasible for any $\varepsilon>0$.

To interpolate between these two situations, we prove the following
theorem, extending the method used in \cite{He} to prove Theorem
\ref{thm:prev}.

For two functions $f,g:\mathbb{N}\rightarrow\mathbb{R}^{+}$ we write
$f=O(g)$ if there exists a constant $C>0$ such that $f(n)\le Cg(n)$
for all $n\in\mathbb{N}$ and $f=o(g)$ if for every $C>0$ the inequaliy
$f(n)\le Cg(n)$ holds for all $n$ sufficiently large. We also write
$f=\Omega(g)$ if $g=O(f)$.
\begin{thm}
\label{thm:main}Let $(h,k)$ be nondecreasing functions $\mathbb{N}\rightarrow\mathbb{R}^{+}$
such that $h(n)=\Omega((\log x)^{1/(1-\log2)})$ and for all sufficiently
large $n$, $k(n)>5$. If they satisfy 
\[
(k(n)-3)\log h(n)\log\log h(n)\geq4\log2\cdot\log n,
\]
for all sufficiently large $n,$ then there exists a $k$-GP-free
sequence $T$ with gaps of size $O(h)$.
\end{thm}
As a corollary, if $k$ is constant we recover Theorem \ref{thm:prev}
with a weaker constant.

\section{Preliminaries}

In this section we generalize the GP-free process of \cite{He} to
probabilistically construct a $k$-GP-free sequence. First we simplify
Theorem \ref{thm:main} by reducing the set of possible length functions
$k$. It suffices to show the following.
\begin{thm}
\label{thm:reduced}If $k$ is a nondecreasing function $\mathbb{N}\rightarrow\{6,8,\ldots\}$
taking on even positive integer values at least $6$, and $h:\mathbb{N}\rightarrow\mathbb{R}^{+}$
is a function satisfying $h(n)=\Omega((\log x)^{1/(1-\log2)})$, $h(n)=o(\sqrt{n})$
and
\[
(k(n)-2)\log h(n)\log\log h(n)\ge4\log2\cdot\log n,
\]
for all $n$ sufficiently large, then there exists a $k$-GP-free
sequence $T$ with gaps of size $O(h)$.\end{thm}
\begin{proof}
(that Theorem \ref{thm:reduced} implies Theorem \ref{thm:main}).
Suppose Theorem \ref{thm:reduced} is true, and let $k$ be as in
Theorem \ref{thm:main}. We can certainly round up $k$ to the nearest
integer to begin with. It is also possible to ignore the finite set
of $n$ for which $k\leq5$, since we only care about $n$ sufficiently
large. If we round $k$ down to the nearest even integer, if it originally
satisfied the inequality of Theorem \ref{thm:main}, then it has decreased
by at most $1$ uniformly, so the inequality above holds. Finally,
if we prove the theorem for all $h(n)=o(\sqrt{n})$, then it follows
for all larger $h$ as well, so we may as well assume $h(n)=o(\sqrt{n})$.
\end{proof}
Let $G_{k}$ be the family of all geometric progressions of positive
integers such that if $t$ is the largest term, then the length is
at least $k(t)$. Enumerate them as $G_{k,i}$ in order lexicographically
as sequences of positive integers. We assume that each $G_{k,i}$
has common ratio $r_{k,i}>1$.

Furthermore, there may be longer $G_{k,i}$ containing shorter ones;
let $G_{k}^{*}$ denote the result of removing from $G_{k}$ all $G_{k,i}$
which contain some $G_{k,j}$ with $j\neq i$. Thus to find a $k$-GP-free
sequence it suffices to construct a sequence $T_{k}$ missing at least
one element from each progression in $G_{k}^{*}$. Let $G_{k,i}^{*}$
denote the $i$-th progression in $G_{k}^{*}$.
\begin{defn}
\label{process}For a nondecreasing function $k:\mathbb{N}\rightarrow\{6,8\ldots\}$,
define the $k$-GP-free process as follows. Define an integer-sequence
valued random variable $U_{k}=(u_{1},u_{2},\ldots)$ where $u_{i}\in G_{k,i}^{*}$
such that if
\[
G_{k,i}^{*}=(a_{i}b_{i}^{k-1},a_{i}b_{i}^{k-2}c_{i},\ldots,a_{i}c_{i}^{k-1}),
\]
then $u_{i}$ is chosen from $a_{i}b_{i}^{k/2-1}c_{i}^{k/2}$ and
$a_{i}b_{i}^{k/2}c_{i}^{k/2-1}$ with equal probability $\frac{1}{2}$.
Each $u_{i}$ is picked independently of the others. Then $T_{k}$
is the random variable whose value is the sequence of all positive
integers never appearing in $U_{k}$, sorted in increasing order.

It is clear that $T_{k}$ is $k$-GP-free by definition, as it misses
at least one term out of each $G_{k,i}^{*}$. We now bound the probability
that a given $n\in\mathbb{N}$ lies in $T_{k}$ generated as above.
For $i,j\ge1$, let $d(n;i,j)$ count the number of ways to factor
$n=ab^{i}c^{j}$ for some $a,b,c\in\mathbb{N}$.\end{defn}
\begin{lem}
\label{lem:individual}For a positive integer $n$, the sequence $T_{k}$
constructed in Definition \ref{process} contains $n$ with probability
\[
\mathbb{P}[T_{k}\ni n]\geq2^{-d(n;k(m)/2,k(m)/2-1)},
\]
where $m$ is any positive integer such that any $G_{k,i}^{*}$ containing
$n$ in its middle two terms has largest term at least $m$.\end{lem}
\begin{proof}
The inequality is equivalent to the statement that $n$ is one of
the middle two terms in at most $d(n;\frac{k(m)}{2},\frac{k(m)}{2}-1)$
progressions of $G_{k}^{*}$. We form an injective correspondence
from progression $G_{k,i}^{*}$ containing $n$ in the middle two
terms to factorizations of $n$ as $n=ab^{k(m)/2}c^{k(m)/2-1}$. If
a progression 
\[
G_{k,i}^{*}=(a_{i}b_{i}^{k'-1},a_{i}b_{i}^{k'-2}c_{i},\ldots,a_{i}c_{i}^{k'-1})
\]
with $b_{i}<c_{i}$ and $k'\geq k(a_{i}c_{i}^{k'-1})$ contains $n$
as one of the middle two terms, then certainly $k(m)\leq k'$. Supposing
$n=a_{i}b_{i}^{k'/2-1}c_{i}^{k'/2}$, we map $G_{k,i}^{*}$ to the
factorization $n=ab^{k(m)/2}c^{k(m)/2-1}$ with $a=a_{i}b_{i}^{(k'-k(m))/2}c_{i}^{(k'-k(m))/2}$,
$b=c_{i}$ and $c=b_{i}$. Similarly if $n=a_{i}b_{i}^{k'/2}c_{i}^{k'/2-1}$
we take $a=a_{i}b_{i}^{(k'-k(m))/2}c_{i}^{(k'-k(m))/2}$, $b=b_{i}$
and $c=c_{i}$. It is easy to see from the assumptions that $b_{i}<c_{i}$
and that no progression in $G_{k}^{*}$ strictly contains another
that the correspondence above is injective, as desired.
\end{proof}
From here we can control the total probability that $T_{k}$ misses
an entire interval of the form $[x,x+Ch(x))$.
\begin{lem}
\label{lem:singlegap}For a gap function $h(x)=o\Big(x^{1-1/(k(x)-1)}\Big)$
and a constant $C>0$, the sequence $T_{k}$ constructed in Definition
\ref{process} satisfies $T_{k}\cap[x,x+Ch(x))=\emptyset$ with probability

\[
\mathbb{P}[T_{k}\cap[x,x+Ch(x))=\emptyset]\leq\exp\Big(-\sum_{n\in[x,x+Ch(x))}\exp\Big(-\log2\cdot d\Big(n;\frac{k(x)}{2},\frac{k(x)}{2}-1\Big)\Big)\Big)
\]
for all $x$ sufficiently large.\end{lem}
\begin{proof}
We first prove that the events $\mathbb{P}[T_{k}\ni n]$ for $n\in[x,x+Ch(x))$
are mutually independent whenever $x$ is sufficiently large. It suffices
to show that no progression in $G_{k}^{*}$ has both middle terms
in the interval. Considering the difference between the two middle
terms in a $G_{k,i}^{*}$, and assuming both lie inside $[x,x+Ch(x))$,
we have 
\begin{eqnarray*}
|a_{i}b_{i}^{k/2-1}c_{i}^{k/2}-a_{i}b_{i}^{k/2}c_{i}^{k/2-1}| & \geq & a_{i}b_{i}^{k/2-1}c_{i}^{k/2-1}\\
 & \geq & x/b_{i}\\
 & \ge & x^{1-1/(k(m)-1)}\\
 & \ge & x^{1-1/(k(x)-1)}
\end{eqnarray*}
where $k\ge k(m)$ depends on the largest term $m=a_{i}c_{i}^{k-1}>x$.
It follows that assuming $h(x)=o\Big(x^{1-1/(k(x)-1)}\Big)$, for
any $C>0$ the middle two terms in any $G_{k,i}^{*}$ with largest
term at most $x$ are further apart than $Ch(x)$ for any $x$ sufficiently
large.

Thus the events corresponding to each $n$ in the interval are mutually
independent, and we can bound the probability involved by a product
\[
\mathbb{P}[T_{k}\cap[x,x+Ch(x))=\emptyset]\leq\prod_{n\in[x,x+Ch(x))}\Big(1-2^{-d(n;k(m)/2,k(m)/2-1)}\Big),
\]
by Lemma \ref{lem:individual}. Since the inequality $1-t\leq e^{-t}$
holds for all real $t$ we arrive at the bound
\[
\mathbb{P}[T_{k}\cap[x,x+Ch(x))=\emptyset]\leq\exp\Big(-\sum_{n\in[x,x+Ch(x))}\exp\Big(-\log2\cdot d(n;\frac{k(m)}{2},\frac{k(m)}{2}-1)\Big)\Big).
\]
Here each $m=m(n)$ can certainly be chosen as any number at most
$n$. Thus we replace them all by $x$, arriving at the desired bound.
\end{proof}
Note that since we assumed $h(x)=o(\sqrt{x})$ the growth condition
in Lemma \ref{lem:singlegap} is automatically satisfied.

\section{Proof of the Main Theorem}

All that remains is to give lower bounds for the sum
\[
S(x,h,k,C)=\sum_{n\in[x,x+Ch)}\exp\Big(-\log2\cdot d\Big(n;\frac{k}{2},\frac{k}{2}-1\Big)\Big),
\]
where $k=k(x)$ and $h=h(x)$ are functions satisfying the conditions
of Theorem \ref{thm:reduced}. To this end we break down $[x,x+Ch)$
into two sets, one of which has few $(k/2-1)$-power divisors, and
restrict the sum to that set.
\begin{lem}
\label{lem:sum-bound}There is a positive constant $B$ independent
of $x$ such that for all sufficiently large $x$,
\[
S(x,h,k,C)\geq BCh(x)\exp\Big(-\log2\exp\Big(\frac{4\log2\cdot\log x}{(k(x)-2)\log h(x)}\Big)\Big).
\]
\end{lem}
\begin{proof}
Fix an $x>0$ and write $k=k(x),h=h(x)$. Denote by $A$ the subset
of $[x,x+Ch)$ consisting of all $n$ divisible by some $p^{k/2-1}$,
where $p\leq h$. We can bound the size of $A$ by
\begin{eqnarray*}
|A| & \leq & \sum_{\mbox{prime }p\leq h}\Big(\frac{Ch}{p^{k/2-1}}+1\Big)\\
 & \leq & (\zeta(k/2-1)-1)Ch+o(h),
\end{eqnarray*}
where $\zeta$ is the Riemann zeta function and we used the elementary
Chebyshev bound $\pi(h)=o(h)$ on the prime-counting function $\pi$.
Since $k\geq6$ and $\zeta(t)-1<1$ uniformly on $t\geq2$, there
exists a constant $B$ such that for $x$, and thus $h$, sufficiently
large, $|A|\leq(1-B)Ch$.

If $n\not\in A$, we can factor $n=p_{1}^{\alpha_{1}}\cdots p_{r}^{\alpha_{r}}n'$
where $n'$ is $(k/2-1)$-th power free, each $\alpha_{i}\geq k/2-1$,
and each $p_{i}\geq h$ is prime. As a result,
\[
\sum_{i}\alpha_{i}\leq\frac{\log n}{\log h},
\]
so by a smoothing argument we can bound $d(n;\frac{k}{2},\frac{k}{2}-1)$
subject to these assumptions,

\[
d\Big(n;\frac{k}{2},\frac{k}{2}-1\Big)\leq\exp\Big(\log2\cdot\frac{\log n}{(k/2-1)\log h}+\log2\cdot\frac{\log n}{(k/2)\log h}\Big),
\]
where we simply bounded the number of pairs $b,c$ satisfying $b^{k/2-1}|n$
and $c^{k/2}|n$. Summing up over all terms in $[x,x+Ch)$ outside
$A$, we get
\[
S(x,h,k,C)\geq BCh\exp\Big(-\log2\exp\Big(\Big(\frac{1}{k}+\frac{1}{k-2}\Big)\frac{(2\log2)\cdot\log x}{\log h}\Big)\Big),
\]
and finally replacing $1/k\le1/(k-2)$ we have the desired inequality.
\end{proof}
Finally, we prove Theorem \ref{thm:reduced} using Lemma \ref{lem:sum-bound}.
\begin{proof}
(of Theorem \ref{thm:reduced}). By Lemma \ref{lem:singlegap} it
suffices to pick $h,k$ such that the sum of probabilities
\[
\sum_{x\ge1}\mathbb{P}[T_{k}\cap[x,x+Ch(x))=\emptyset]\le\sum_{x\geq1}\exp(-S(x,h,k,C))<1
\]
for $C$ sufficiently large, forcing the probability of finding a
$T$ with gaps $O(h)$ to be nonzero. This will hold as long as the
sum converges for some fixed $C$; making $C$ large enough will make
the sum arbitrarily small. Now, suppose that $(k-2)\log h\log\log h\geq4\log2\cdot\log n$
as in Theorem \ref{thm:reduced}. Then, applying the inequality of
Lemma \ref{lem:sum-bound}, we have
\begin{eqnarray*}
S(x,h,k,C) & \geq & BCh\exp(-\log2\log h)\\
 & \geq & BCh^{1-\log2},
\end{eqnarray*}
and finally since $h=\Omega((\log x)^{1/(1-\log2)})$, we get
\[
\sum_{x\geq1}\exp(-S(x,h,k,C))\le\sum_{x\ge1}x^{-BCD},
\]
for some constant $D>0$, so picking $C$ for which $BC>1$ gives
a convergent sum.
\end{proof}

\section{Closing Remarks}

The goal of this paper was to interpolate smoothly between the two
feasible pairs $(h,k)=(\exp(C\log N/\log\log N),6)$ and $(h,k)=(1,\log N/\log2)$,
and we recover both pairs, up to constants, in the relation
\[
(k(n)-3)\log h(n)\log\log h(n)\geq4\log2\cdot\log n.
\]

Unfortunely, when $k$ is sufficiently close to $\log n$, then the
method of Theorem \ref{thm:main} fails because $h=o((\log x)^{1/(1-\log2)})$.
Nevertheless, we expect all pairs $(h,k)$ which satisfy this inequality
to be feasible. In the case that $h=1$ we can make an improvement
on $(1,\log N/\log2)$.
\begin{prop}
For any $\varepsilon>0$, if $k(n)=\varepsilon\log n$ then there
exists a $k$-GP-free sequence $T$ with gaps of size $O(1)$.\end{prop}
\begin{proof}
We say a positive integer $m$ is divisible by a $k$-th power if
$p^{\lceil k(m)\rceil}|m$ for some prime $p$, and that $m$ is $k$-free
otherwise. Consider the sequence $T$ of all $k$-free integers; we
claim that its gaps are uniformly bounded. In fact, note that if $p^{\lceil k(m)\rceil}|m$
then
\begin{eqnarray*}
p^{k(m)} & \le & m\\
\varepsilon\log m\cdot\log p & \le & \log m\\
\log p & \le & \frac{1}{\varepsilon},
\end{eqnarray*}
and so $p$ lies in the finite set of all primes less than $e^{1/\varepsilon}$.
In particular, for $x$ sufficiently large, the interval $[x,x+e^{1/\varepsilon}+1)$
will contain at least one $k$-free number. Indeed, it is easy to
check that each $p\le e^{1/\varepsilon}$ contributes at most one
multiple of $p^{k(x)}$ to that interval.
\end{proof}
Further improvement in the case of $h$ small or constant along these
lines is blocked by the Chinese Remainder Theorem. In particular,
for $k=o(\log n)$ and any constant $h$ we can find infinitely many
intervals $[x,x+h)$ in which each positive integer in $[x,x+h)$
is divisible by arbitrarily many $k(x)$-th powers of primes.

The probabilistic method in Definition \ref{process} is by no means
optimal, but is defined in such a way to guarantee the independence
of events in an interval $[n,n+Ch)$. We expect that a sophisticated
study of redundancies in our method can substantially improve at least
the constant in Theorem \ref{thm:main}.

\section{Acknowledgements}

I would like to thank Levent Alpoge and Joe Gallian for correcting
many mistakes.

\end{document}